\documentclass[10pt,psamsfonts]{amsart}
\usepackage{amsmath}
\usepackage{amsthm}
\usepackage{amssymb}
\usepackage{amscd}
\usepackage{amsfonts}
\usepackage{amsbsy}
\usepackage{graphicx}
\usepackage[dvips]{psfrag}
\usepackage{array}
\usepackage{color}
\usepackage{epsfig}
\usepackage{url}
\usepackage{overpic}
\usepackage{tikz-cd}
\usepackage{enumitem}
\usepackage{hyperref}
\usepackage{soul}
\usepackage{listings}

\newtheorem {theorem} {Theorem}
\newtheorem {proposition} [theorem]{Proposition}

\usepackage{tabularx}

\newcommand{\bg}{{\bf g}}

\newcommand{\f}{{\bf f}}

\usepackage{mathpazo}

\title[On the higher order stroboscopic averaged functions]{Addendum to\\ Higher order stroboscopic averaged functions:\\ a general relationship with Melnikov functions}

\author[D.D. Novaes]
{Douglas D. Novaes$^{1}$}

\address{$^1$ Departamento de Matem\'{a}tica - Instituto de Matem\'{a}tica, Estat\'{i}stica e Computa\c{c}\~{a}o Cient\'{i}fica (IMECC) - Universidade
Estadual de Campinas (UNICAMP),  Rua S\'{e}rgio Buarque de Holanda, 651, Cidade Universit\'{a}ria Zeferino Vaz, 13083-859, Campinas, SP,
Brazil} \email{ddnovaes@unicamp.br}



\begin{document}

\maketitle

\begin{abstract}
This addendum presents a relevant stronger consequence of the main theorem of the paper "Higher order stroboscopic averaged functions: a general relationship with Melnikov functions", EJQTDE No. 77 (2021).
\end{abstract}

This addendum addresses the findings presented in the paper \cite{novaes21} titled "Higher order stroboscopic averaged functions: a general relationship with Melnikov functions" published in EJQTDE No. 77 (2021).

The main result of the referred paper,  \cite[Theorem A]{novaes21}, establishes a general relationship between averaged functions $\bg_i$ and Melnikov functions $\f_i$. As a direct consequence of this general relationship,  \cite[Corollary A]{novaes21}  states that if, for some $\ell\in\{2,\ldots,k\}$, either $\f_1=\cdots=\f_{\ell-1}=0$ or $\bg_1=\cdots=\bg_{\ell-1}=0$, then $\f_i=T\,\bg_i$ for $i\in\{1,\ldots,\ell\}$. This consequence was somewhat expected based on existing results in the literature within more restricted contexts. Here, we will demonstrate that under the same conditions, the relationship $\f_i=T\,\bg_i$ actually holds for every $i\in\{1,\ldots,2\ell-1\}$, which represents a more unexpected outcome. The expressions for $\bg_{2\ell}(z)$ and $\f_{2\ell}(z)$ will also be provided.

\begin{proposition}
Let $\ell\in\{2,\ldots,k\}$. If either $\f_1=\cdots=\f_{\ell-1}=0$ or $\bg_1=\cdots=\bg_{\ell-1}=0,$ then $\f_i=T\,\bg_i$ for $i\in\{1,\ldots,2\ell-1\}$, and 
\[
\bg_{2\ell}(z)=\dfrac{1}{T}\left( \f_{2\ell}(z)- \dfrac{1}{2} d \f_{\ell}(z)\cdot \f_{\ell}(z)\right) \,\, \text{or, equivalently,}\,\, \f_{2\ell}(z)=T\bg_{2\ell}(z)+\dfrac{T^2}{2} d \bg_{\ell}(z)\cdot \bg_{\ell}(z).
\]
\end{proposition}
\begin{proof}
Given $\ell\in\{2,\ldots,k\}$, assume that either $\f_1=\cdots=\f_{\ell-1}=0$ or $\bg_1=\cdots=\bg_{\ell-1}=0$. From  \cite[Corollary A]{novaes21}, we have that 
\begin{equation}\label{relcor}
\bg_i=\f_i=0, \,\, \text{for}\,\, i\in\{1,\ldots,\ell-1\},\,\,\text{and}\,\, \bg_{\ell}=\dfrac{1}{T}\f_{\ell}.
\end{equation}

For any $i$, \cite[Theorem A]{novaes21} provides
\begin{equation}\label{rec:gi}
\begin{aligned}
\bg_i(z)=&\dfrac{1}{T}\left( \f_i(z)-\sum_{j=1}^{i-1}\sum_{m=1}^j\dfrac{1}{j!}d^m \bg_{i-j}(z) \int_0^T B_{j,m}\big(\tilde y_1,\ldots,\tilde y_{j-m+1}\big)(s,z)ds\right),
\end{aligned}
\end{equation}
where $\tilde y_i(t,z),$ for $i\in\{1,\ldots,k\},$ are polynomial in the variable $t$ recursively defined as follows:
\begin{equation}\label{tildeyi}
\begin{aligned}
\tilde y_1(t,z)=&t\, \bg_1(z)\vspace{0.3cm}\\
\tilde y_i(t,z)=& i!t\,\bg_i(z) +\sum_{j=1}^{i-1}\sum_{m=1}^j\dfrac{i!}{j!}d^m \bg_{i-j}(z) \int_0^t B_{j,m}\big(\tilde y_1,\ldots,\tilde y_{j-m+1}\big)(s,z)ds.
\end{aligned}
\end{equation}

Taking \eqref{relcor} into account, the function $\bg_{i-j}$ in \eqref{rec:gi} vanishes for $i-j\leq \ell-1$, that is, for $j\geq i-\ell+1$. Thus,
\begin{equation}\label{rec:gi2}
\begin{aligned}
\bg_i(z)=&\dfrac{1}{T}\left( \f_i(z)-\sum_{j=1}^{i-\ell}\sum_{m=1}^j\dfrac{1}{j!}d^m \bg_{i-j}(z) \int_0^T B_{j,m}\big(\tilde y_1,\ldots,\tilde y_{j-m+1}\big)(s,z)ds\right).
\end{aligned}
\end{equation}
Also, from \eqref{relcor} and \eqref{tildeyi}, one has 
\begin{equation}\label{tildeyi2}
\tilde y_1=\cdots\tilde y_{\ell-1}=0\,\,\text{and}\,\, \tilde y_{\ell}(t,z)=\ell! t \bg_{\ell}(z)=\dfrac{\ell!}{T} t \f_{\ell}(z).
\end{equation}

Now, let $i\in\{\ell+1,\ldots,2\ell-1\}.$  Thus, for $ j\leq i-\ell$ and $m\geq 1$, one has that
\[
j-m+1\leq i-\ell\leq 2\ell-1-\ell=\ell-1,
\]
which implies, from \eqref{tildeyi2}, that $\tilde y_1=\cdots=\tilde y_{j-m+1}=0$ in \eqref{rec:gi2}. Consequently, $\f_i(z)=T \bg_i(z)$.

Finally, from \eqref{rec:gi2},
\begin{equation}\label{rec:g2l}
\begin{aligned}
\bg_{2\ell}(z)=&\dfrac{1}{T}\left( \f_{2\ell}(z)-\sum_{j=1}^{\ell}\sum_{m=1}^j\dfrac{1}{j!}d^m \bg_{2\ell-j}(z) \int_0^T B_{j,m}\big(\tilde y_1,\ldots,\tilde y_{j-m+1}\big)(s,z)ds\right).
\end{aligned}
\end{equation}
Notice that, for $1\leq j\leq \ell$ and $1\leq m\leq j$, the relationship $j-m+1\geq \ell$ implies that $m=1$ and $j=\ell$, which are the only possible values for $m$ and $j$ for which $\tilde y_{j-m+1}$ in \eqref{rec:g2l} is not vanishing. In this case, from \eqref{tildeyi2}, $\tilde y_1=\cdots=\tilde y_{j-m}=0$ and $\tilde y_{j-m+1}=\tilde y_{\ell}=\ell! t \bg_{\ell}(z)$. Thus,
\[
\begin{aligned}
\bg_{2\ell}(z)=&\dfrac{1}{T}\left( \f_{2\ell}(z)-\dfrac{1}{\ell !}d \bg_{\ell}(z) \int_0^T B_{\ell,1}\big(0,\ldots,0,\ell! t \bg_{\ell}(z)\big)\,ds\right)\\
=&\dfrac{1}{T}\left( \f_{2\ell}(z)-\dfrac{1}{\ell !}d \bg_{\ell}(z) \int_0^T \ell! t \bg_{\ell}(z)\,ds\right)\\
=&\dfrac{1}{T}\left( \f_{2\ell}(z)- \dfrac{T^2}{2} d \bg_{\ell}(z)\cdot \bg_{\ell}(z)\right)=\dfrac{1}{T}\left( \f_{2\ell}(z)- \dfrac{1}{2} d \f_{\ell}(z)\cdot \f_{\ell}(z)\right).
\end{aligned}
\]
Equivalently,
\[
\f_{2\ell}(z)=T\bg_{2\ell}(z)+\dfrac{T^2}{2} d \bg_{\ell}(z)\cdot \bg_{\ell}(z).
\]
\end{proof}

\section*{Acknowledgements}
The author is partially supported by S\~{a}o Paulo Research Foundation (FAPESP) grants 2021/10606-0, 2018/13481-0, and 2019/10269-3, and by Conselho Nacional de Desenvolvimento Cient\'{i}fico e Tecnol\'{o}gico (CNPq) grant 309110/2021-1.


\begin{thebibliography}{99}

\bibitem{novaes21}
\textsc{D.~D. Novaes},
Higher order stroboscopic averaged functions: a general relationship with Melnikov functions,
  \textit{Electronic Journal of Qualitative Theory of Differential Equations}, 77:1--9, 2021.  \url{https://doi.org/10.14232/ejqtde.2021.1.77}.


\end{thebibliography}
\end{document}